\documentclass[11pt,reqno]{amsart}

\usepackage{amsmath,amssymb,amsthm,mathtools}
\usepackage{geometry}
\usepackage{enumitem}
\usepackage{booktabs}
\usepackage{xcolor}
\usepackage[colorlinks=true,linkcolor=blue,citecolor=blue,urlcolor=blue]{hyperref}
\usepackage{aliascnt}
\usepackage[capitalise,noabbrev]{cleveref}
\usepackage{fancyhdr}

\hypersetup{
  pdftitle={Improved Weyl bounds on short intervals},
  pdfauthor={Xiyu Hu},
  pdfsubject={Short rational Weyl sums, large values, and Diophantine applications},
  pdfcreator={pdflatex}
}
\setlist{itemsep=0pt,topsep=0.3em,parsep=0pt,partopsep=0pt}
\allowdisplaybreaks
\numberwithin{equation}{section}
\newtheorem{theorem}{Theorem}[section]

\newaliascnt{proposition}{theorem}
\newtheorem{proposition}[proposition]{Proposition}
\aliascntresetthe{proposition}

\newaliascnt{lemma}{theorem}
\newtheorem{lemma}[lemma]{Lemma}
\aliascntresetthe{lemma}

\newaliascnt{corollary}{theorem}
\newtheorem{corollary}[corollary]{Corollary}
\aliascntresetthe{corollary}

\theoremstyle{definition}
\newaliascnt{definition}{theorem}

\aliascntresetthe{definition}

\theoremstyle{remark}
\newaliascnt{remark}{theorem}
\newtheorem{remark}[remark]{Remark}
\aliascntresetthe{remark}

\crefname{theorem}{Theorem}{Theorems}
\Crefname{theorem}{Theorem}{Theorems}
\crefname{proposition}{Proposition}{Propositions}
\Crefname{proposition}{Proposition}{Propositions}
\crefname{lemma}{Lemma}{Lemmas}
\Crefname{lemma}{Lemma}{Lemmas}
\crefname{corollary}{Corollary}{Corollaries}
\Crefname{corollary}{Corollary}{Corollaries}
\crefname{definition}{Definition}{Definitions}
\Crefname{definition}{Definition}{Definitions}
\crefname{remark}{Remark}{Remarks}
\Crefname{remark}{Remark}{Remarks}

\newcommand{\R}{\mathbb R}
\newcommand{\Z}{\mathbb Z}
\newcommand{\T}{\mathbb T}
\newcommand{\e}{\mathrm e}

\newcommand{\cD}{\mathcal D}

\newcommand{\cM}{\mathcal M}

\newcommand{\cR}{\mathcal R}
\newcommand{\cT}{\mathcal T}

\newcommand{\Spec}{\operatorname{Spec}}

\title{Improved Weyl bounds on short intervals}
\author{Xiyu Hu}
\address{School of Mathematical Sciences, University of Chinese Academy of Sciences}
\email{hxyqpr@gmail.com}
\date{}

\begin{document}

\begin{abstract}
For an integer $d\ge3$, put
$\Delta_d=\min\{2^{d-1},d(d-1)\}$.  Let $a/q$ be reduced, let
\[
  P(X)=\frac aqX^d+\alpha_{d-1}X^{d-1}+\cdots+\alpha_0,
\]
and let $\mathcal I$ be an interval of $H\le q$ consecutive integers.  We prove
\[
  \left|\sum_{n\in\mathcal I}e(P(n))\right|
  \ll_{d,\varepsilon}
  q^{1/d}H^\varepsilon+H^{1-1/\Delta_d+\varepsilon}.
\]
Consequently, for every prime $p>d$, every degree-$d$ polynomial
$P\in\mathbb F_p[X]$, and every interval $\mathcal I$ of $H$ consecutive
integers with $p^{1/d}<H<p^{1/(d-1)}$, writing $H^d/p=H^u$, one has
\[
  \left|\sum_{n\in\mathcal I}e_p(P(n))\right|
  \ll_{d,\varepsilon}
  H^{1-\min\{u/d,\,1/\Delta_d\}+\varepsilon}.
\]
At the level of the displayed power-saving exponents, this improves throughout
the full natural short-interval window the canonical uniform pointwise
benchmark obtained by taking the stronger of classical Weyl differencing and
the optimal-VMVT consequence.
\end{abstract}

\maketitle
\tableofcontents

\section{Introduction}

\subsection{Weyl sums and the scope of known bounds}

For $d\ge3$ and $z\in\mathbb R$, write
\[
  e(z)=\exp(2\pi iz),
  \qquad
  \|z\|=\operatorname{dist}(z,\mathbb Z).
\]
Given
\[
  \boldsymbol\alpha=(\alpha_1,\ldots,\alpha_d)\in\mathbb T^d,
  \qquad M\in\mathbb Z,
  \qquad H\in\mathbb N,
\]
the associated Weyl sum over a translated interval is
\begin{equation}\label{eq:intro-general-weyl-sum}
  W_d(\boldsymbol\alpha;M,H)
  =
  \sum_{M<n\le M+H}
  e(\alpha_1n+\cdots+\alpha_dn^d).
\end{equation}
Weyl introduced these sums in his study of uniform distribution modulo one
\cite{Weyl1916}.  They are now basic tools in Diophantine approximation, the
circle method, Waring's problem, estimates for the Riemann zeta-function, and
related problems; see, for example,
\cite{Baker1986,Vaughan1997} and \cite{IwaniecKowalski2004}.

Suppose that $(a,q)=1$ and
\[
  \left|\alpha_d-\frac aq\right|\le q^{-2}.
\]
Classical Weyl differencing gives
\begin{equation}\label{eq:intro-classical-weyl}
  |W_d(\boldsymbol\alpha;M,H)|
  \ll_{d,\varepsilon}
  H^{1+\varepsilon}
  \left(q^{-1}+H^{-1}+qH^{-d}\right)^{2^{1-d}};
\end{equation}
see \cite{Vaughan1997}.  The modern theory culminated in proofs of the main
conjecture in the Vinogradov mean value theorem, through efficient
congruencing and decoupling; see
\cite{Wooley2012,BourgainDemeterGuth2016,Wooley2019}, as well as
\cite{Bourgain2017} for an account of the resulting Weyl estimates.  A
standard consequence, conveniently recorded by Heath-Brown
\cite{HeathBrown2017}, is
\begin{equation}\label{eq:intro-standard-weyl}
  |W_d(\boldsymbol\alpha;M,H)|
  \ll_{d,\varepsilon}
  H^{1+\varepsilon}
  \left(q^{-1}+H^{-1}+qH^{-d}\right)^{1/[d(d-1)]}.
\end{equation}
Both bounds are uniform in the remaining coefficients and in $M$.

It will be convenient to put
\begin{equation}\label{eq:Delta-definition}
  \Delta_d=\min\{2^{d-1},d(d-1)\}.
\end{equation}
In the nontrivial range, taking the stronger of
\eqref{eq:intro-classical-weyl} and \eqref{eq:intro-standard-weyl} gives the
canonical general uniform pointwise benchmark
\begin{equation}\label{eq:intro-generic-weyl}
  |W_d(\boldsymbol\alpha;M,H)|
  \ll_{d,\varepsilon}
  H^{1+\varepsilon}
  \left(q^{-1}+H^{-1}+qH^{-d}\right)^{1/\Delta_d}.
\end{equation}
Among these two bounds, classical Weyl differencing is stronger in degrees
$3,4,5$, while the Vinogradov-mean-value exponent is stronger from degree
six onwards.

The word ``general'' is important here.  There is no single best Weyl
exponent without specifying which coefficients are fixed, which estimates
must be uniform, and what arithmetic or metric hypotheses are imposed.
Refinements of Weyl's inequality and of the passage from mean values to
pointwise estimates were obtained, in particular, by Heath-Brown and by
Robert and Sargos \cite{HeathBrown1988,RobertSargos2000,HeathBrown2017}.
Such estimates can improve the canonical bounds above in specific low-degree
or restricted Diophantine-approximation regimes.  In degree eight, Parsell
obtained a further refinement of an estimate of Robert and Sargos, yielding
improvements over other Weyl-type bounds in certain restricted denominator
ranges \cite{Parsell2014}.

Substantially stronger estimates are also known for special classes of
phases.  Modulo a prime, Kerr obtained estimates for incomplete Gauss sums
\cite{Kerr2018}.  For the sparse cubic phase $an^3/q+\gamma n$ with general
rational modulus $q$, he improved Weyl differencing in a particular short
range \cite{KerrCubic2021}.  For a quartic monomial with a quadratic
irrational coefficient, Heath-Brown obtained an exponent stronger than the
classical quartic Weyl exponent \cite{HeathBrownQuartic2024}.  These results
are not uniform in arbitrary lower-order coefficients and are therefore not
directly comparable with the problem considered here.  A very recent
preprint of Mirzoabdughafurov studies a uniform pointwise estimate for
monomial short Weyl sums in an intermediate rational-approximation range
arising in Waring's problem with almost proportional summands
\cite{Mirzoabdughafurov2026}; its phase class and parameter range are
different from those considered below.

A different and especially relevant body of work relaxes pointwise
uniformity in some coefficients.  Wooley proved strong bounds holding for
almost all choices of a collection of intermediate coefficients, uniformly
in the remaining coefficients \cite{WooleyPerturbations2016}.  Chen and
Shparlinski subsequently combined completion, continuity, and
self-improvement arguments to obtain sharper estimates for Weyl sums with
partially prescribed coefficients \cite{ChenShparlinskiNewBounds2021}.  Of
particular relevance to translated intervals is their theorem that, for
almost every $x_d\in\mathbb T$,
\begin{equation}\label{eq:intro-chen-shparlinski-short}
  \sup_{\boldsymbol y\in\mathbb T^{d-1}}
  \sup_{M\in\mathbb Z}
  \left|
  \sum_{M<n\le M+N}
  e(y_1n+\cdots+y_{d-1}n^{d-1}+x_dn^d)
  \right|
  \le N^{1-1/(d+1)+o(1)}
  \qquad(N\to\infty).
\end{equation}
For $d>3$, the saving $1/(d+1)$ is much larger than the
$1/[d(d-1)]$ saving obtained by applying
\eqref{eq:intro-standard-weyl} directly.  The quantifiers in
\eqref{eq:intro-chen-shparlinski-short}, however, are fundamentally
different from those in this paper: an almost-everywhere assertion does not
furnish a bound for a prescribed rational value $x_d=a/q$.

More generally, Chen, Kerr, Maynard, and Shparlinski showed that the
square-root scale occurs for a full-measure set of coefficient vectors along
infinitely many lengths \cite{ChenKerrMaynardShparlinski2023}.  Chen and
Shparlinski studied mean values and almost-everywhere estimates when the
coefficient vector is restricted to a measure space, including spheres,
moment curves, and line segments
\cite{ChenShparlinskiRestricted2021}.  These results describe typical or
averaged behaviour rather than a pointwise estimate uniform on a prescribed
rational fibre.

One can obtain still better degree dependence by changing the summation set.
For Weyl sums supported on smooth numbers, estimates of Vaughan, developed
further by Wooley, have recently led, through work of Br\"udern and Wooley,
to minor-arc exponents of order $1/(d\log d)$
\cite{Vaughan1989,Wooley1995Smooth,BruedernWooley2025}.  Such estimates have
much stronger degree dependence than \eqref{eq:intro-standard-weyl}, but the
summation is over smooth integers rather than over all integers in a
consecutive interval.

Finally, for a prime $p>d$, write $e_p(z)=e(z/p)$.  Completion followed by
the Weil bound gives, for every degree-$d$ polynomial $P\in\mathbb F_p[X]$
and every interval $\mathcal I\subseteq\mathbb Z$ of length $H\le p$,
\begin{equation}\label{eq:intro-completion-weil}
  \left|\sum_{n\in\mathcal I}e_p(P(n))\right|
  \ll_d p^{1/2}\log p;
\end{equation}
see, for example, \cite{IwaniecKowalski2004}.  This is powerful above the
square-root scale, up to the logarithmic loss, but does not give a nontrivial
saving near $H=p^{1/d}$.  Thus \eqref{eq:intro-generic-weyl} is not an absolute upper
envelope for every class of Weyl sums; it is the natural uniform pointwise
benchmark for arbitrary lower coefficients in the short rational setting
considered below.

\subsection{Short rational Weyl sums}

Let
\[
  P(X)=\frac aqX^d+\alpha_{d-1}X^{d-1}+\cdots+\alpha_0,
  \qquad (a,q)=1,
\]
and let $\mathcal I\subset\mathbb Z$ be an interval of $H$ consecutive
integers.  The main object of this paper is the short rational Weyl sum
\begin{equation}\label{eq:intro-rational-sum}
  S_{\mathcal I}(P)=\sum_{n\in\mathcal I}e(P(n)).
\end{equation}
The scale $H=q^{1/d}$ is a genuine obstruction to uniform cancellation.
Indeed, for the phase $P(n)=n^d/q$, if $H=o(q^{1/d})$, then
$n^d/q=o(1)$ uniformly for $n\le H$, and hence
\[
  \sum_{n\le H}e(n^d/q)=(1+o(1))H.
\]
Thus one cannot expect a uniform power saving below this threshold.  Our
principal short-interval result is the following.

\begin{theorem}\label{thm:rational-short-intro}
Let $d\ge3$ be fixed, let $(a,q)=1$, and let
\[
  P(X)=\frac aqX^d+\alpha_{d-1}X^{d-1}+\cdots+\alpha_0
  \in\mathbb R[X].
\]
If $\mathcal I\subset\mathbb Z$ is an interval of $H$ consecutive integers
with $1\le H\le q$, then, for every $\varepsilon>0$,
\begin{equation}\label{eq:rational-short-intro}
  |S_{\mathcal I}(P)|
  \ll_{d,\varepsilon}
  q^{1/d}H^\varepsilon
  +H^{1-1/\Delta_d+\varepsilon}.
\end{equation}
\end{theorem}

Now let $p>d$ be prime.  Choosing integer lifts of the coefficients gives the
finite-field form.

\begin{corollary}\label{cor:prime-short-intro}
Let $p>d$ be prime, let $P\in\mathbb F_p[X]$ have degree $d\ge3$, and let
$\mathcal I$ be an interval of $H\le p$ consecutive integers.  Then, for
every $\varepsilon>0$,
\begin{equation}\label{eq:prime-short-intro}
  \left|\sum_{n\in\mathcal I}e_p(P(n))\right|
  \ll_{d,\varepsilon}
  p^{1/d}H^\varepsilon
  +H^{1-1/\Delta_d+\varepsilon}.
\end{equation}
\end{corollary}

Recent work of Koh and Shparlinski \cite{KohShparlinski2026} develops
finite-field analogues of mean-value and restricted-mean-value results for
short rational exponential sums.  Their estimates average over coefficient
families and use Mordell-type arguments in place of the Vinogradov mean value
theorem.  By contrast, \cref{cor:prime-short-intro} is an individual
estimate, uniform in the lower coefficients and in the location of the
interval.

Consider now the natural short-interval window
\begin{equation}\label{eq:natural-window-intro}
  p^{1/d}<H<p^{1/(d-1)}.
\end{equation}
Its lower endpoint reflects the obstruction above, while its upper endpoint
is the transition at which $pH^{-d}$ and $H^{-1}$ are equal.  Define
$u\in(0,1)$ by
\begin{equation}\label{eq:u-definition-intro}
  \frac{H^d}{p}=H^u.
\end{equation}
Throughout this window, $pH^{-d}=H^{-u}$ dominates the other terms in
\eqref{eq:intro-generic-weyl}, so the canonical general bound gives
\begin{equation}\label{eq:intro-generic-u}
  \left|\sum_{n\in\mathcal I}e_p(P(n))\right|
  \ll_{d,\varepsilon}
  H^{1-u/\Delta_d+\varepsilon}.
\end{equation}
On the other hand, \cref{cor:prime-short-intro} gives
\begin{equation}\label{eq:prime-u-intro}
  \left|\sum_{n\in\mathcal I}e_p(P(n))\right|
  \ll_{d,\varepsilon}
  H^{1-\tau_d(u)+\varepsilon},
  \qquad
  \tau_d(u)=\min\left\{\frac ud,\frac1{\Delta_d}\right\}.
\end{equation}
Since $\Delta_d>d$, for every fixed $0<u<1$ one has
\[
  \tau_d(u)>\frac{u}{\Delta_d}.
\]
Thus, at the level of the displayed power-saving exponents, the new estimate
improves the canonical general bound throughout the full window
\eqref{eq:natural-window-intro}.  For $d\ge6$, this is precisely an
improvement over the optimal-VMVT Weyl exponent.  For $d=3,4,5$, the relevant
comparison is instead with classical Weyl differencing.  The conclusion is
uniform in all lower-order coefficients, so it is distinct from the metric
estimate \eqref{eq:intro-chen-shparlinski-short}, as well as from bounds for
monomial, sparse, or smooth-number-supported sums.

For example, when $d=6$ and $H=\lfloor p^{1/6+\delta}\rfloor$, one has
$\Delta_6=30$ and
\begin{equation}\label{eq:degree-six-intro}
  \left|\sum_{n\in\mathcal I}e_p(P(n))\right|
  \ll_{\varepsilon}
  \begin{cases}
    p^{1/6}H^\varepsilon,&0<\delta\le1/174,\\[1mm]
    H^{29/30+\varepsilon},&1/174\le\delta<1/30.
  \end{cases}
\end{equation}

\subsection{The inverse theorem behind the short-interval estimate}

For $k\ge2$, $\boldsymbol\alpha=(\alpha_1,\ldots,\alpha_k)\in\mathbb R^k$
(viewed modulo $\mathbb Z^k$), and $N\ge1$, put
\begin{equation}\label{eq:intro-weyl-sum}
  g_k(\boldsymbol\alpha;N)
  =\sum_{n=1}^{N}e(\alpha_1n+\cdots+\alpha_kn^k).
\end{equation}
A second strand of the theory concerns the arithmetic structure forced by an
unusually large Weyl sum.  Baker's inverse results show that a large value of
\eqref{eq:intro-weyl-sum} forces simultaneous rational approximation of all
the coefficients with a controlled common denominator; see
\cite{Baker1982,Baker1992,Baker1986,Baker2016}.  In particular, Baker
\cite[Theorem~4]{Baker2016} obtained such a denominator for a general
degree-$k$ polynomial under the threshold
\[
  |g_k(\boldsymbol\alpha;N)|
  >N^{1-1/[2k(k-1)]+\varepsilon}.
\]
When all intermediate coefficients vanish, he obtained the stronger
threshold $1/[k(k-1)]$.  The main analytic input of the present paper removes
this sparsity restriction.

Put
\begin{equation}\label{eq:Kk-definition}
  K_k=k(k-1).
\end{equation}

\begin{theorem}\label{thm:inverse-main}
Let $k\ge3$ be fixed and let $\varepsilon>0$.  Suppose that
$\boldsymbol\alpha\in\mathbb R^k$, $N$ is sufficiently large, and
\begin{equation}\label{eq:inverse-hypothesis}
  |g_k(\boldsymbol\alpha;N)|\ge A
  >N^{1-1/K_k+\varepsilon}.
\end{equation}
Then there are integers $q,a_1,\ldots,a_k$ such that
\begin{equation}\label{eq:inverse-q}
  1\le q\ll_{k,\varepsilon}
  N^\varepsilon(NA^{-1})^k
\end{equation}
and
\begin{equation}\label{eq:inverse-coefficients}
  |q\alpha_j-a_j|
  \ll_{k,\varepsilon}
  N^{-j+\varepsilon}(NA^{-1})^k
  \qquad(1\le j\le k).
\end{equation}
\end{theorem}

Combining \cref{thm:inverse-main} with the classical inverse theorem obtained
by Weyl differencing gives the following convenient form.

\begin{corollary}\label{cor:combined-inverse}
Let $k\ge3$ be fixed and let $\varepsilon>0$.  Suppose that
$\boldsymbol\alpha\in\mathbb R^k$, $N$ is sufficiently large, and
\begin{equation}\label{eq:combined-threshold}
  |g_k(\boldsymbol\alpha;N)|\ge A
  >N^{1-1/\Delta_k+\varepsilon}.
\end{equation}
Then the conclusions
\eqref{eq:inverse-q}--\eqref{eq:inverse-coefficients} hold.
\end{corollary}

Baker--Chen--Shparlinski \cite[Lemma~2.6]{BakerChenShparlinski2024} record
the previous large-value parameter as
\[
  \min\{2^{k-1},2k(k-1)\}.
\]
The first strict improvement occurs at $k=6$:
\[
\begin{array}{c|ccccc}
  k&4&5&6&7&8\\ \hline
  \text{previous denominator}&8&16&32&64&112\\
  \Delta_k&8&16&30&42&56.
\end{array}
\]
For large $k$, the denominator governing the large-value threshold is
reduced from $2k(k-1)$ to $k(k-1)$.

\subsection{Diophantine and finite-field applications}

The inverse theorem also improves two classical Diophantine applications.
The first is the least fractional part of a general polynomial.

\begin{theorem}\label{thm:small-fractional-parts-intro}
Let $k\ge3$, let $N\ge1$, and let
$\alpha_1,\ldots,\alpha_k\in\mathbb R$.  For every $\varepsilon>0$,
\begin{equation}\label{eq:small-fractional-parts-intro}
  \min_{1\le n\le N}
  \left\|\alpha_kn^k+\cdots+\alpha_1n\right\|
  \ll_{k,\varepsilon}N^{-1/\Delta_k+\varepsilon}.
\end{equation}
In particular, for $k\ge6$ the exponent is $1/[k(k-1)]$.
\end{theorem}

For $k\ge6$, this extends to arbitrary intermediate coefficients the
exponent that Baker obtained for binomial phases.  Related work of Yeon
\cite{Yeon2024} develops mean-value estimates and small-fractional-part
results for sparse collections of degrees and for sums of several
polynomials.  Those results concern a different structural regime from the
single general polynomial in \cref{thm:small-fractional-parts-intro}.

The second comes from Baker's Harman-sieve argument for polynomial values at
prime arguments \cite{BakerPrimes2017}.

\begin{theorem}\label{thm:prime-fractional-parts-intro}
Let $k\ge4$, and let $f\in\mathbb R[X]$ have degree $k$ and irrational
leading coefficient.  For every
\begin{equation}\label{eq:prime-nu-range}
  0<\nu<\frac{0.4079}{\Delta_k},
\end{equation}
there are infinitely many primes $\ell$ such that
\begin{equation}\label{eq:prime-fractional-parts}
  \|f(\ell)\|<\ell^{-\nu}.
\end{equation}
\end{theorem}

For $k=4,5$ this reproduces Baker's bound.  It improves his
general-polynomial denominator from $32$ to $30$ in degree six, from $64$ to
$42$ in degree seven, and from $2k(k-1)$ to $k(k-1)$ for every $k\ge8$.

On the finite-field side, \cref{cor:prime-short-intro} yields quantitative
equidistribution, a bound for the longest cyclic interval omitted by the short
polynomial image, and an additive-basis criterion for sums of these values.
These consequences are stated and proved in
\cref{sec:distribution-applications}.

\subsection{The mechanism and organization}

The proof of \cref{thm:inverse-main} uses all cuts of one large full-length
sum.  At each cut either the prefix or the tail remains large.  After
reversing prefixes, this gives linearly many large partial sums whose leading
coefficient is fixed and whose lower coefficient vectors lie on one integer
translation orbit.  The critical Vinogradov mean value theorem in degree
$k-1$ gives a maximal $K_k$th-moment estimate on this fixed-leading-coefficient
fibre.  An anisotropic band-limited sampling inequality then forces two orbit
points to collide at the canonical coefficient scales $N^{-j}$.  The
translation action is triangular, so the difference of the two translation
parameters supplies a simultaneous preliminary denominator.  Baker's
denominator-compression lemma reduces it to the stated scale $(NA^{-1})^k$.

Sections~2--5 prove the inverse theorem.  Section~6 proves the short rational
and finite-field Weyl estimates.  Section~7 develops discrepancy,
omitted-interval, and additive-basis consequences.  Section~8 treats small
fractional parts over the integers and over the primes.  Section~9 records
briefly where the improved threshold enters related large-value arguments of
Baker, Chen, Shparlinski, and Brandes.

Throughout, the degree is fixed.  Implied constants may depend on the degree
and on displayed small parameters, but on no other quantities.

\section{Preliminaries}

\subsection{Translation of polynomial coefficients}

For $\boldsymbol\alpha=(\alpha_1,\ldots,\alpha_k)\in\T^k$, write
\[
  P_{\boldsymbol\alpha}(X)=\sum_{j=1}^{k}\alpha_jX^j.
\]
For an integer $m$, define the translation map $\cT_m:\T^k\to\T^k$ by
\begin{equation}\label{eq:translation-action}
  (\cT_m\boldsymbol\alpha)_j
  =\sum_{\ell=j}^{k}
  \binom{\ell}{j}\alpha_\ell m^{\ell-j}
  \pmod1
  \qquad(1\le j\le k).
\end{equation}
Then
\begin{equation}\label{eq:translated-polynomial}
  P_{\boldsymbol\alpha}(m+X)
  =P_{\cT_m\boldsymbol\alpha}(X)
   +P_{\boldsymbol\alpha}(m)
  \pmod1.
\end{equation}
The maps form an integer unipotent action:
\begin{equation}\label{eq:translation-group-law}
  \cT_m\cT_n=\cT_{m+n},
  \qquad
  \cT_m^{-1}=\cT_{-m}.
\end{equation}
Let $\pi:\T^k\to\T^{k-1}$ denote projection onto the first $k-1$ coordinates, and let
\begin{equation}\label{eq:reflection}
  \cR(\alpha_1,\ldots,\alpha_k)
  =((-1)^1\alpha_1,\ldots,(-1)^k\alpha_k).
\end{equation}
Thus the coefficients of $P_{\boldsymbol\alpha}(m-X)$ are $\cR\cT_m\boldsymbol\alpha$, apart from the constant term.

\subsection{The Vinogradov mean value theorem}

For integers $s,r\ge1$, let $J_{s,r}(N)$ be the number of solutions of
\[
  x_1^j+\cdots+x_s^j
  =y_1^j+\cdots+y_s^j
  \qquad(1\le j\le r)
\]
with $1\le x_i,y_i\le N$.  Equivalently,
\[
  J_{s,r}(N)
  =\int_{\T^r}
  \left|\sum_{n=1}^{N}
  \e(\beta_1n+\cdots+\beta_rn^r)
  \right|^{2s}
  d\boldsymbol\beta.
\]
The main conjecture in the Vinogradov mean value theorem was proved independently by Bourgain--Demeter--Guth and Wooley.  It states that
\begin{equation}\label{eq:vmvt}
  J_{s,r}(N)
  \ll_{r,\varepsilon}
  N^{s+\varepsilon}
  +N^{2s-r(r+1)/2+\varepsilon}.
\end{equation}
See \cite{BourgainDemeterGuth2016,Wooley2019}.  We use the critical case
\begin{equation}\label{eq:critical-vmvt}
  s=\frac{r(r+1)}2,
  \qquad
  J_{s,r}(N)\ll_{r,\varepsilon}N^{s+\varepsilon}.
\end{equation}

\subsection{Baker's denominator-compression lemma}

We quote the following lemma in the form used in Baker's proof of \cite[Theorem~4]{Baker2016}.  It is a restatement of \cite[Lemma~4.6]{Baker1986}.

\begin{lemma}[Baker]\label{lem:baker-compression}
Let $k\ge3$, and suppose that there are integers $r,v_2,\ldots,v_k$ satisfying
\[
  \gcd(r,v_2,\ldots,v_k)=1
\]
and
\begin{equation}\label{eq:baker-pre-approx}
  |r\alpha_j-v_j|
  \le \frac{N^{1-j}}{4k^4}
  \qquad(2\le j\le k).
\end{equation}
If, for some $\eta>0$,
\begin{equation}\label{eq:baker-amplitude}
  |g_k(\boldsymbol\alpha;N)|\ge A
  >r^{1-1/k}N^\eta,
\end{equation}
then there is an integer $t$ with $1\le t\le2k^2$ such that
\begin{align}
  tr&\le (NA^{-1})^kN^\eta,\label{eq:baker-compressed-q}\\
  t|r\alpha_j-v_j|
  &\le (NA^{-1})^kN^{-j+\eta}
  \qquad(2\le j\le k),\label{eq:baker-compressed-high}\\
  \|tr\alpha_1\|
  &\le (NA^{-1})N^{-1+\eta}.
  \label{eq:baker-compressed-linear}
\end{align}
\end{lemma}

The key point for us is that the preliminary denominator $r$ need only satisfy the canonical accuracy \eqref{eq:baker-pre-approx}.  It need not already have the final size $(NA^{-1})^k$.

\section{A maximal critical moment on a fixed-leading-coefficient fibre}

Fix $k\ge3$ and put $K=K_k=k(k-1)$.  For $\theta\in\T$, $\boldsymbol\beta=(\beta_1,\ldots,\beta_{k-1})\in\T^{k-1}$, and $1\le L\le N$, write
\begin{equation}\label{eq:fibre-sum}
  F_{\theta}(\boldsymbol\beta;L)
  =\sum_{n=1}^{L}
  \e\left(\theta n^k+\sum_{j=1}^{k-1}\beta_jn^j\right)
\end{equation}
and
\begin{equation}\label{eq:maximal-fibre-sum}
  \cM_{\theta,N}(\boldsymbol\beta)
  =\max_{1\le L\le N}
  |F_{\theta}(\boldsymbol\beta;L)|.
\end{equation}

\begin{theorem}\label{thm:maximal-fibre-moment}
For every fixed $k\ge3$, every $\varepsilon>0$, and every $\theta\in\T$,
\begin{equation}\label{eq:maximal-fibre-moment}
  \int_{\T^{k-1}}
  \cM_{\theta,N}(\boldsymbol\beta)^K
  d\boldsymbol\beta
  \ll_{k,\varepsilon}
  N^{K/2+\varepsilon}.
\end{equation}
The implied constant is uniform in $\theta$.
\end{theorem}

\begin{proof}
Let $s=K/2=k(k-1)/2$, which is the critical exponent for the Vinogradov system of degree $k-1$.  First consider an integer interval $J$ of length $R\le N$.  Expanding the $2s$th moment and integrating in $\beta_1,\ldots,\beta_{k-1}$ gives
\begin{align*}
  &\int_{\T^{k-1}}
  \left|\sum_{n\in J}
  \e\left(\theta n^k+\sum_{j=1}^{k-1}\beta_jn^j\right)
  \right|^{2s}
  d\boldsymbol\beta\\
  &\qquad=
  \sum_{\substack{x_1,\ldots,x_s,y_1,\ldots,y_s\in J\\
  \sum_i x_i^j=\sum_i y_i^j\ (1\le j\le k-1)}}
  \e\left(\theta\left(\sum_i x_i^k-\sum_i y_i^k\right)\right).
\end{align*}
Taking absolute values and translating $J$ to an interval beginning at one, the right-hand side is at most $J_{s,k-1}(R)$.  Translation preserves the system of equal power sums because the numbers of variables on the two sides are equal.  Hence \eqref{eq:critical-vmvt} gives
\begin{equation}\label{eq:interval-fibre-moment}
  \int_{\T^{k-1}}
  \left|\sum_{n\in J}
  \e\left(\theta n^k+\sum_{j=1}^{k-1}\beta_jn^j\right)
  \right|^{K}
  d\boldsymbol\beta
  \ll_{k,\varepsilon}R^{K/2+\varepsilon}.
\end{equation}

Every initial interval $[1,L]$ is the disjoint union of at most $1+\log_2N$ dyadic intervals from a fixed dyadic grid $\cD_N$.  Therefore
\[
  \cM_{\theta,N}(\boldsymbol\beta)^K
  \ll_k
  (\log N)^{K-1}
  \sum_{J\in\cD_N}
  \left|\sum_{n\in J}
  \e\left(\theta n^k+\sum_{j=1}^{k-1}\beta_jn^j\right)
  \right|^K.
\]
At dyadic scale $R$ there are $O(N/R)$ intervals.  Integrating and using \eqref{eq:interval-fibre-moment}, we obtain
\[
  \int_{\T^{k-1}}\cM_{\theta,N}^K
  \ll_{k,\varepsilon}
  N^\varepsilon
  \sum_{R\text{ dyadic}\le N}
  \frac NR R^{K/2+\varepsilon}
  \ll_{k,\varepsilon}N^{K/2+2\varepsilon}.
\]
Renaming $2\varepsilon$ as $\varepsilon$ proves the theorem.
\end{proof}

\begin{remark}\label{rem:maximal-decoupling}
The only deep input in \cref{thm:maximal-fibre-moment} is the critical Vinogradov mean value estimate in degree $k-1$.  One may therefore view the theorem as a maximal fixed-leading-coefficient consequence of sharp moment-curve decoupling.
\end{remark}

\section{An anisotropic sampling inequality}

The next result is a vector-valued sampling inequality for trigonometric polynomials.  It allows the polynomial attached to the sampling point to vary, provided all polynomials are dominated by one maximal function.

For positive integers $L_1,\ldots,L_r$, define the anisotropic metric
\begin{equation}\label{eq:anisotropic-metric}
  d_{\boldsymbol L}(\boldsymbol x,\boldsymbol y)
  =\max_{1\le j\le r}
  L_j\|x_j-y_j\|
  \qquad(\boldsymbol x,\boldsymbol y\in\T^r).
\end{equation}
A finite set $X\subset\T^r$ is called $\delta$-separated if
$d_{\boldsymbol L}(\boldsymbol x,\boldsymbol y)\ge\delta$ for distinct $\boldsymbol x,\boldsymbol y\in X$.

\begin{lemma}\label{lem:reproducing-kernel}
For every integer $M\ge2$ and $L\ge1$, there is a periodic kernel $\Psi_L$ such that
\begin{enumerate}[label=\textup{(\roman*)}]
\item $\widehat{\Psi_L}(n)=1$ for every integer $|n|\le L$;
\item $\|\Psi_L\|_{L^1(\T)}\ll_M1$;
\item
\[
  |\Psi_L(x)|
  \ll_M L(1+L\|x\|)^{-M}
  \qquad(x\in\T).
\]
\end{enumerate}
\end{lemma}

\begin{proof}
Choose a smooth compactly supported function $\widehat\psi$ on $\R$ that equals one on $[-1,1]$ and is supported in $[-2,2]$.  Periodise the inverse Fourier transform at scale $L$, or equivalently set
\[
  \Psi_L(x)=\sum_{n\in\Z}\widehat\psi(n/L)\e(nx).
\]
Poisson summation and rapid decay of the inverse Fourier transform give (ii) and (iii), while (i) is immediate.
\end{proof}

\begin{proposition}[Anisotropic sampling]\label{prop:anisotropic-sampling}
Let $p\ge1$ and $L_1,\ldots,L_r\ge1$.  Suppose that
$\boldsymbol x_1,\ldots,\boldsymbol x_M\in\T^r$ are $\delta$-separated with respect to \eqref{eq:anisotropic-metric}.  For each $1\le\nu\le M$, let $f_\nu$ be a trigonometric polynomial satisfying
\begin{equation}\label{eq:spectrum-box}
  \Spec(f_\nu)
  \subseteq
  \prod_{j=1}^{r}[-L_j,L_j]\cap\Z^r.
\end{equation}
Then
\begin{equation}\label{eq:anisotropic-sampling}
  \sum_{\nu=1}^{M}|f_\nu(\boldsymbol x_\nu)|^p
  \ll_{r,p,\delta}
  \left(\prod_{j=1}^{r}L_j\right)
  \int_{\T^r}
  \sup_{1\le\nu\le M}|f_\nu(\boldsymbol y)|^p
  d\boldsymbol y.
\end{equation}
\end{proposition}

\begin{proof}
Let
\[
  \Psi_{\boldsymbol L}(\boldsymbol x)
  =\prod_{j=1}^{r}\Psi_{L_j}(x_j),
\]
where the kernels in \cref{lem:reproducing-kernel} are chosen with a sufficiently large decay exponent, depending only on $r$.  By \eqref{eq:spectrum-box},
$f_\nu=f_\nu*\Psi_{\boldsymbol L}$.  H\"older's inequality and the uniform $L^1$ bound for the kernel give
\begin{equation}\label{eq:kernel-holder}
  |f_\nu(\boldsymbol x_\nu)|^p
  \ll_{r,p}
  \int_{\T^r}
  |f_\nu(\boldsymbol y)|^p
  |\Psi_{\boldsymbol L}(\boldsymbol x_\nu-\boldsymbol y)|
  d\boldsymbol y.
\end{equation}

It remains to sum the kernels.  After rescaling the $j$th coordinate by $L_j$, the $\delta$-separation condition implies that every unit cube contains $O_{r,\delta}(1)$ of the rescaled points.  The rapid product decay in \cref{lem:reproducing-kernel} therefore yields, uniformly in $\boldsymbol y$,
\begin{equation}\label{eq:kernel-packing}
  \sum_{\nu=1}^{M}
  |\Psi_{\boldsymbol L}(\boldsymbol x_\nu-\boldsymbol y)|
  \ll_{r,\delta}
  \prod_{j=1}^{r}L_j.
\end{equation}
Summing \eqref{eq:kernel-holder}, replacing $|f_\nu(\boldsymbol y)|$ by the pointwise supremum, and using \eqref{eq:kernel-packing} proves \eqref{eq:anisotropic-sampling}.
\end{proof}

\begin{remark}\label{rem:sampling-vs-stability}
A direct Lipschitz argument would only show that a value of size $A=N/B$ persists pointwise on boxes of side lengths $(BN^j)^{-1}$.  The sampling inequality works at the larger canonical scales $N^{-j}$: it uses a local weighted $L^p$ lower bound rather than pointwise persistence.  This removes the factor $B^{-(k-1)}$ which would otherwise lead to the weaker denominator $k^2-1$.
\end{remark}

\section{Translation-orbit collisions and the inverse theorem}

\subsection{All cuts of one large sum}

\begin{lemma}\label{lem:all-cuts}
Let $k\ge3$, let $N\ge1$, and set
\[
  A=|g_k(\boldsymbol\alpha;N)|.
\]
There exist a sign type $\sigma\in\{+,-\}$, a real number
$\theta\in\{\alpha_k,(-1)^k\alpha_k\}$, a set $X$ of integers with
\begin{equation}\label{eq:X-size}
  |X|\gg_k N,
\end{equation}
contained in an interval of length at most $N/(2k!)$, and integers
$1\le L_x\le N$ for $x\in X$, such that
\begin{equation}\label{eq:large-translated-partial}
  |F_\theta(\boldsymbol\beta(x);L_x)|\ge A/2
  \qquad(x\in X).
\end{equation}
Moreover, either
\begin{equation}\label{eq:forward-centres}
  \boldsymbol\beta(x)=\pi(\cT_x\boldsymbol\alpha)
  \qquad(x\in X)
\end{equation}
or
\begin{equation}\label{eq:backward-centres}
  \boldsymbol\beta(x)=\pi(\cR\cT_x\boldsymbol\alpha)
  \qquad(x\in X).
\end{equation}
\end{lemma}

\begin{proof}
Let
\[
  S_m=\sum_{n=1}^{m}\e(P_{\boldsymbol\alpha}(n)),
  \qquad 0\le m\le N.
\]
Since $S_N=S_m+(S_N-S_m)$, for every $m$ at least one of
$|S_m|$ and $|S_N-S_m|$ is at least $A/2$.

If the tail is large, then by \eqref{eq:translated-polynomial}
\[
  S_N-S_m
  =\e(P_{\boldsymbol\alpha}(m))
  F_{\alpha_k}(\pi(\cT_m\boldsymbol\alpha);N-m).
\]
If the prefix is large, reverse its order and put $x=m+1$:
\[
  S_m
  =\sum_{h=1}^{m}\e(P_{\boldsymbol\alpha}(x-h))
  =\e(P_{\boldsymbol\alpha}(x))
  F_{(-1)^k\alpha_k}(\pi(\cR\cT_x\boldsymbol\alpha);m).
\]
After assigning each cut to one of the two alternatives, one alternative occurs for at least $(N+1)/2$ cuts.  Partition the corresponding parameter range into $4k!$ consecutive intervals.  One interval contains $\gg_kN$ selected parameters and, for sufficiently large $N$, has length at most $N/(2k!)$.  Discarding a possible zero-length partial sum proves the lemma; bounded $N$ is absorbed into the implied constants.
\end{proof}

\subsection{A collision at the canonical coefficient scales}

\begin{proposition}\label{prop:orbit-collision}
Let $k\ge3$ and $K=k(k-1)$.  For every sufficiently small fixed
$\delta=\delta(k)>0$ and every $\varepsilon>0$, the following holds for all sufficiently large $N$.  If
\begin{equation}\label{eq:collision-amplitude}
  |g_k(\boldsymbol\alpha;N)|
  >N^{1-1/K+\varepsilon},
\end{equation}
then, for the set of centres supplied by \cref{lem:all-cuts}, there are distinct $x,y\in X$ such that
\begin{equation}\label{eq:centre-collision}
  \|\beta_j(x)-\beta_j(y)\|
  <\delta N^{-j}
  \qquad(1\le j\le k-1).
\end{equation}
\end{proposition}

\begin{proof}
Suppose instead that the centres are $\delta$-separated for the anisotropic scales
\[
  L_j=N^j,
  \qquad 1\le j\le k-1.
\]
For each $x\in X$, set
\[
  f_x(\boldsymbol\beta)
  =F_\theta(\boldsymbol\beta;L_x).
\]
The Fourier support of $f_x$ is contained in
\[
  [0,N]\times[0,N^2]\times\cdots\times[0,N^{k-1}],
\]
and
\[
  \sup_{x\in X}|f_x(\boldsymbol\beta)|
  \le \cM_{\theta,N}(\boldsymbol\beta).
\]
Apply \cref{prop:anisotropic-sampling} with $p=K$.  Since
\[
  \prod_{j=1}^{k-1}N^j=N^{K/2},
\]
\cref{lem:all-cuts,thm:maximal-fibre-moment} give
\[
  |X|(A/2)^K
  \ll_{k,\delta}
  N^{K/2}
  \int_{\T^{k-1}}\cM_{\theta,N}(\boldsymbol\beta)^K
  d\boldsymbol\beta
  \ll_{k,\eta}N^{K+\eta}
\]
for every $\eta>0$.  Since $|X|\gg_kN$, this implies
\[
  A\ll_{k,\eta}N^{1-1/K+\eta/K}.
\]
Choosing $\eta<K\varepsilon/2$ contradicts \eqref{eq:collision-amplitude} for sufficiently large $N$.
\end{proof}

\subsection{Extracting a common denominator from a collision}

We first isolate an elementary triangular-algebra lemma.

\begin{lemma}\label{lem:triangular-extraction}
Let $k\ge3$, let $1\le |h|\le N$, and let
$\alpha_2,\ldots,\alpha_k\in\R$.  Suppose that, for some $\eta>0$,
\begin{equation}\label{eq:relative-translation-approx}
  \left\|
  \sum_{\ell=j+1}^{k}
  \binom{\ell}{j}\alpha_\ell h^{\ell-j}
  \right\|
  \le \eta N^{-j}
  \qquad(1\le j\le k-1).
\end{equation}
Then there are integers $v_2,\ldots,v_k$ such that
\begin{equation}\label{eq:triangular-output}
  |k!h\alpha_\ell-v_\ell|
  \ll_k \eta N^{1-\ell}
  \qquad(2\le\ell\le k).
\end{equation}
\end{lemma}

\begin{proof}
Put $x_c=h\alpha_{c+1}$ for $1\le c\le k-1$.  Define the upper-triangular integer matrix $A_h=(a_{j,c})_{1\le j,c\le k-1}$ by
\[
  a_{j,c}=
  \begin{cases}
    \displaystyle\binom{c+1}{j}h^{c-j},&c\ge j,\\
    0,&c<j.
  \end{cases}
\]
Then the left side of \eqref{eq:relative-translation-approx} is the distance of the $j$th coordinate of $A_h\boldsymbol x$ from an integer.  The diagonal entries of $A_h$ are $2,3,\ldots,k$, so
\begin{equation}\label{eq:Ah-det}
  \det A_h=k!.
\end{equation}
Moreover, with $D_h=\operatorname{diag}(h,h^2,\ldots,h^{k-1})$,
\[
  A_h=D_h^{-1}A_1D_h.
\]
Consequently the $(c,j)$ entry of $A_h^{-1}$ vanishes unless $j\ge c$, and in that case is $O_k(|h|^{j-c})$.  The same statement, multiplied by $k!$, holds for the integer matrix $\operatorname{adj}(A_h)$.

Choose $\boldsymbol z\in\Z^{k-1}$ and an error vector $\boldsymbol e$ with
\[
  A_h\boldsymbol x=\boldsymbol z+\boldsymbol e,
  \qquad |e_j|\le\eta N^{-j}.
\]
Multiplying by the adjugate gives
\[
  k!\boldsymbol x
  =\operatorname{adj}(A_h)\boldsymbol z
   +\operatorname{adj}(A_h)\boldsymbol e.
\]
The first term is integral.  For the $c$th coordinate of the error term,
\[
  \left|(\operatorname{adj}(A_h)\boldsymbol e)_c\right|
  \ll_k
  \sum_{j=c}^{k-1}|h|^{j-c}\eta N^{-j}
  \ll_k\eta N^{-c}.
\]
Since $c=\ell-1$, this is \eqref{eq:triangular-output}.
\end{proof}

\begin{lemma}\label{lem:collision-to-approximation}
Under the hypotheses of \cref{prop:orbit-collision}, let $x\ne y$ satisfy \eqref{eq:centre-collision}, and put $h=x-y$.  Then, if $\delta=\delta(k)$ is sufficiently small, there are integers $r,v_2,\ldots,v_k$ such that
\begin{equation}\label{eq:preliminary-r-size}
  1\le r\le N,
  \qquad
  \gcd(r,v_2,\ldots,v_k)=1,
\end{equation}
and
\begin{equation}\label{eq:preliminary-r-approx}
  |r\alpha_j-v_j|
  \le\frac{N^{1-j}}{4k^4}
  \qquad(2\le j\le k).
\end{equation}
\end{lemma}

\begin{proof}
In the forward case \eqref{eq:forward-centres}, the collision means that the first $k-1$ coordinates of
$\cT_x\boldsymbol\alpha-\cT_y\boldsymbol\alpha$ are within the scales in \eqref{eq:centre-collision} of integers.  In the backward case, apply the integer diagonal isometry $\cR$ first and reach the same conclusion.

Apply the integer map $\cT_{-y}$.  By \eqref{eq:translation-group-law}, the resulting difference is
$\cT_h\boldsymbol\alpha-\boldsymbol\alpha$.  Since $|y|\le N+1$, the triangular form of $\cT_{-y}$ and \eqref{eq:centre-collision} imply
\begin{equation}\label{eq:relative-after-conjugation}
  \left\|(\cT_h\boldsymbol\alpha-\boldsymbol\alpha)_j\right\|
  \ll_k\delta N^{-j}
  \qquad(1\le j\le k-1).
\end{equation}
Indeed, an error of size $N^{-\ell}$ in the $\ell$th coordinate is multiplied by at most $O_k(N^{\ell-j})$ when it reaches the $j$th coordinate.

Expanding \eqref{eq:relative-after-conjugation} gives the hypotheses of \cref{lem:triangular-extraction} with $\eta\ll_k\delta$.  Hence $k!h\alpha_j$ has an integral approximant with error $O_k(\delta N^{1-j})$.  By \cref{lem:all-cuts},
\[
  |h|\le \frac{N}{2k!}.
\]
Set $r_0=k!|h|\le N/2$, changing the signs of the approximating integers when necessary.  Choose $\delta(k)>0$ sufficiently small that the resulting errors are at most $N^{1-j}/(4k^4)$.  Finally divide $r_0$ and all approximating integers by their common greatest divisor.  The denominator decreases and the errors do not increase, giving \eqref{eq:preliminary-r-size}--\eqref{eq:preliminary-r-approx}.
\end{proof}

\subsection{Proof of the main theorem}

\begin{proof}[Proof of \cref{thm:inverse-main}]
Let
\[
  B=NA^{-1}\ge1.
\]
Choose a small auxiliary parameter $\eta>0$, depending on $k$ and the displayed $\varepsilon$.  By \cref{prop:orbit-collision,lem:collision-to-approximation}, there are integers $r,v_2,\ldots,v_k$ satisfying \eqref{eq:preliminary-r-size}--\eqref{eq:preliminary-r-approx}.

We verify the amplitude hypothesis in \cref{lem:baker-compression}.  Since $r\le N$,
\[
  r^{1-1/k}N^\eta
  \le N^{1-1/k+\eta}.
\]
For $k\ge3$,
\[
  \frac1k-\frac1{k(k-1)}
  =\frac{k-2}{k(k-1)}>0.
\]
Thus, after taking $\eta$ sufficiently small and then $N$ sufficiently large, \eqref{eq:inverse-hypothesis} implies
\[
  A>r^{1-1/k}N^\eta.
\]
Apply \cref{lem:baker-compression}.  With $q=tr$ and $a_j=tv_j$ for $2\le j\le k$, equations \eqref{eq:baker-compressed-q} and \eqref{eq:baker-compressed-high} give
\[
  q\le B^kN^\eta,
  \qquad
  |q\alpha_j-a_j|
  \le B^kN^{-j+\eta}
  \quad(2\le j\le k).
\]
Choose $a_1\in\Z$ nearest to $q\alpha_1$.  By \eqref{eq:baker-compressed-linear},
\[
  |q\alpha_1-a_1|
  \le BN^{-1+\eta}
  \le B^kN^{-1+\eta}.
\]
Since the auxiliary parameter may be chosen smaller than the stated $\varepsilon$, this proves \eqref{eq:inverse-q}--\eqref{eq:inverse-coefficients}.
\end{proof}

\begin{proof}[Proof of \cref{cor:combined-inverse}]
The conclusion under the threshold $A>N^{1-2^{1-k}+\varepsilon}$ is the classical inverse theorem for Weyl sums; see Baker \cite{Baker1982,Baker1992} or Baker--Chen--Shparlinski \cite[Lemma~2.6]{BakerChenShparlinski2024}.  Taking the stronger of that result and \cref{thm:inverse-main} gives \eqref{eq:combined-threshold}.
\end{proof}

\section{Improved Weyl bounds on short intervals}\label{sec:short-interval-bounds}

\subsection{A rational leading coefficient}

\begin{proof}[Proof of \cref{thm:rational-short-intro}]
Translate $\mathcal I$ to $[1,H]$.  This changes only the lower coefficients, while the leading coefficient remains $a/q$.  Let
\[
  S=\sum_{n=1}^{H}e(P(n)),
  \qquad A=|S|.
\]
Fix $\eta>0$ so small that $3d\eta<\varepsilon$.  If
\[
  A\le H^{1-1/\Delta_d+3\eta},
\]
there is nothing to prove after enlarging the final $H^\varepsilon$ loss.  Otherwise apply \cref{cor:combined-inverse} with auxiliary loss $\eta$.  There is an integer $r$ satisfying
\begin{equation}\label{eq:rational-r}
  1\le r\ll(HA^{-1})^dH^\eta
\end{equation}
and
\begin{equation}\label{eq:rational-leading-approx}
  \left\|r\frac aq\right\|
  \ll(HA^{-1})^dH^{-d+\eta}
  =A^{-d}H^\eta.
\end{equation}
The lower bound on $A$ gives
\[
  r\ll H^{d/\Delta_d-(3d-1)\eta}<H\le q
\]
for sufficiently large $H$, since $d/\Delta_d<1$.  As $(a,q)=1$ and $1\le r<q$,
\[
  \left\|r\frac aq\right\|\ge\frac1q.
\]
Together with \eqref{eq:rational-leading-approx}, this gives
$A\ll q^{1/d}H^{\eta/d}$.  Combining the two cases and renaming the small parameter proves \eqref{eq:rational-short-intro}.  Bounded $H$ is absorbed into the implied constant.
\end{proof}

\begin{proof}[Proof of \cref{cor:prime-short-intro}]
Choose integer representatives for the coefficients of $P$.  After division by $p$, the leading coefficient is $a_d/p$ with $(a_d,p)=1$.  Apply \cref{thm:rational-short-intro} with $q=p$.
\end{proof}

\subsection{Comparison with the classical and VMVT bounds}

Taking the better of \eqref{eq:intro-classical-weyl} and
\eqref{eq:intro-standard-weyl}, and using
\eqref{eq:natural-window-intro}, the dominant term is $pH^{-d}=H^{-u}$.
Hence the standard generic estimate is
\begin{equation}\label{eq:standard-u}
  \left|\sum_{n\in\mathcal I}e_p(P(n))\right|
  \ll_{d,\varepsilon}
  H^{1-u/\Delta_d+\varepsilon}.
\end{equation}
On the other hand $p=H^{d-u}$, and therefore
$p^{1/d}=H^{1-u/d}$.  This proves \eqref{eq:prime-u-intro}.  If the first
branch of $\tau_d(u)$ is active, then
\[
  \frac ud>\frac{u}{\Delta_d},
\]
because $\Delta_d>d$ for $d\ge3$.  If the second branch is active, then
\[
  \frac1{\Delta_d}>\frac{u}{\Delta_d},
\]
since $u<1$.  Thus the new estimate is strictly stronger throughout the full
window, including the low degrees in which classical Weyl differencing is
stronger than the VMVT consequence.

Writing
\begin{equation}\label{eq:H-delta}
  H=\left\lfloor p^{1/d+\delta}\right\rfloor,
  \qquad 0<\delta<\frac1{d(d-1)},
\end{equation}
one has, up to a harmless $o(1)$ in exponent calculations,
\begin{equation}\label{eq:u-delta}
  u=\frac{d^2\delta}{1+d\delta}.
\end{equation}
The two branches meet at
\begin{equation}\label{eq:delta-branch-general}
  \delta_{\mathrm{br}}(d)=\frac1{d(\Delta_d-1)}.
\end{equation}
Consequently
\begin{equation}\label{eq:delta-piecewise-general}
  \left|\sum_{n\in\mathcal I}e_p(P(n))\right|
  \ll_{d,\varepsilon}
  \begin{cases}
    p^{1/d}H^\varepsilon,
      &0<\delta\le\dfrac1{d(\Delta_d-1)},\\[3mm]
    H^{1-1/\Delta_d+\varepsilon},
      &\dfrac1{d(\Delta_d-1)}\le\delta<\dfrac1{d(d-1)}.
  \end{cases}
\end{equation}
For $d\ge6$, $\Delta_d=d(d-1)$.  Hence the plateau is
$H^{1-1/[d(d-1)]+\varepsilon}$, the branch point is
\[
  \delta_{\mathrm{br}}(d)
  =\frac1{d(d(d-1)-1)}
  =\frac1{d^3}+\frac1{d^4}+O(d^{-5}),
\]
and the strict improvement persists up to the endpoint
$1/[d(d-1)]$ of the natural window.  Formula \eqref{eq:degree-six-intro} follows by substituting $d=6$ and $\Delta_6=30$.

\section{Distribution and additive applications}\label{sec:distribution-applications}

\subsection{Discrepancy of short polynomial images}

Let $p$ be prime, let $P\in\mathbb F_p[X]$, and let $\mathcal I$ have length $H\le p$.  For an interval $J\subset[0,1)$ define
\[
  N_J(P;\mathcal I)
  =\#\left\{n\in\mathcal I:
  \frac{\widetilde{P(n)}}p\in J\right\},
\]
where $\widetilde{P(n)}\in\{0,1,\ldots,p-1\}$ is the standard representative, and put
\begin{equation}\label{eq:discrepancy-definition}
  D(P;\mathcal I)
  =\sup_{J\subset[0,1)}
  \left|N_J(P;\mathcal I)-H|J|\right|.
\end{equation}

\begin{corollary}\label{cor:discrepancy}
Let $p>d$ be prime and let $P\in\mathbb F_p[X]$ have degree $d\ge3$.  Then
\begin{equation}\label{eq:discrepancy-bound}
  D(P;\mathcal I)
  \ll_{d,\varepsilon}
  \left(p^{1/d}+H^{1-1/\Delta_d}\right)H^\varepsilon.
\end{equation}
\end{corollary}

\begin{proof}
Set
\[
  E=p^{1/d}+H^{1-1/\Delta_d}.
\]
The Erd\H{o}s--Tur\'an inequality \cite[Chapter~1]{Montgomery1994} gives, for $1\le M<p$,
\[
  D(P;\mathcal I)
  \ll\frac HM+
  \sum_{h=1}^{M}\frac1h
  \left|\sum_{n\in\mathcal I}e_p(hP(n))\right|.
\]
Since $h<p$, the polynomial $hP$ still has degree $d$.  By \cref{cor:prime-short-intro}, each inner sum is
$O_{d,\varepsilon}(EH^\varepsilon)$.  If $E\ge H/2$ the claimed estimate is trivial; otherwise choose $M=\lfloor H/E\rfloor<p$.  Then
\[
  D(P;\mathcal I)\ll EH^\varepsilon\log(2H),
\]
and the logarithm is absorbed into $H^\varepsilon$.
\end{proof}

The gain is particularly transparent just above the natural threshold.  If
$H=p^{1/d+\delta}$ and
$0<\delta<1/[d(\Delta_d-1)]$, then the first branch is active and
\begin{equation}\label{eq:normalized-discrepancy-first-branch}
  \frac{D(P;\mathcal I)}H
  \ll_{d,\varepsilon}p^{-\delta+\varepsilon}.
\end{equation}
Using the standard generic estimate \eqref{eq:standard-u} in the same
Erd\H{o}s--Tur\'an argument gives instead
$p^{-d\delta/\Delta_d+\varepsilon}$.  Thus the exponent of the saving is
larger by a factor $\Delta_d/d$; for $d\ge6$ this factor is $d-1$.

\subsection{The longest omitted interval of residues}

Identify intervals in $\mathbb F_p$ with cyclic intervals of consecutive residues.  Let $G(P;\mathcal I)$ be the largest length of a cyclic interval containing none of the values $P(n)$ with $n\in\mathcal I$.

\begin{corollary}\label{cor:gaps}
Under the hypotheses of \cref{cor:discrepancy},
\begin{equation}\label{eq:gap-bound}
  G(P;\mathcal I)
  \ll_{d,\varepsilon}
  \frac pH
  \left(p^{1/d}+H^{1-1/\Delta_d}\right)H^\varepsilon.
\end{equation}
\end{corollary}

\begin{proof}
If a cyclic interval of $Y$ residues contains no value $P(n)$, split it at zero if necessary.  For one of the resulting ordinary intervals, the expected number of hits is comparable to $HY/p$, while the actual number is zero.  Hence
$HY/p\ll D(P;\mathcal I)$.  Applying \eqref{eq:discrepancy-bound} and solving for $Y$ proves the result.
\end{proof}

In the first branch $H=p^{1/d+\delta}$, this becomes
\begin{equation}\label{eq:gap-first-branch}
  G(P;\mathcal I)\ll_{d,\varepsilon}p^{1-\delta+\varepsilon}.
\end{equation}
The corresponding consequence of the standard generic Weyl bound is
$p^{1-d\delta/\Delta_d+\varepsilon}$; for $d\ge6$ this is
$p^{1-\delta/(d-1)+\varepsilon}$.

\subsection{An additive-basis consequence}

For $s\ge1$ and $a\in\mathbb F_p$, let
\begin{equation}\label{eq:representation-function}
  R_s(a)=\#\left\{(n_1,\ldots,n_s)\in\mathcal I^s:
  P(n_1)+\cdots+P(n_s)=a\right\}.
\end{equation}
Thus $R_s(a)>0$ for every $a$ means that the short polynomial image, counted with its natural parametrisation, is an additive basis of order $s$ for $\mathbb F_p$.

Put
\begin{equation}\label{eq:E-def-additive}
  E_d(p,H)=p^{1/d}+H^{1-1/\Delta_d}.
\end{equation}

\begin{corollary}\label{cor:additive-basis}
Fix $d\ge3$, $s\ge1$, and $\eta>0$.  Under the hypotheses of \cref{cor:prime-short-intro}, suppose that
\begin{equation}\label{eq:additive-basis-condition}
  \left(\frac{H}{E_d(p,H)}\right)^s
  \ge pH^\eta.
\end{equation}
Then, for sufficiently large $H$ in terms of the fixed parameters,
\[
  R_s(a)>0\qquad(a\in\mathbb F_p).
\]
\end{corollary}

\begin{proof}
Fourier inversion gives
\begin{align}\label{eq:Rs-fourier}
  R_s(a)
  &=\frac1p\sum_{t\in\mathbb F_p}
  e_p(-ta)
  \left(\sum_{n\in\mathcal I}e_p(tP(n))\right)^s\\
  &=\frac{H^s}{p}
  +O\left(
  \max_{t\in\mathbb F_p^\times}
  \left|\sum_{n\in\mathcal I}e_p(tP(n))\right|^s
  \right).\nonumber
\end{align}
For $t\ne0$, the polynomial $tP$ still has degree $d$.  Choose the $\varepsilon$ in \cref{cor:prime-short-intro} so small that $s\varepsilon<\eta/2$.  Then the error in \eqref{eq:Rs-fourier} is
\[
  \ll_{d,s,\eta}E_d(p,H)^sH^{\eta/2}.
\]
Condition \eqref{eq:additive-basis-condition} makes this $o(H^s/p)$, after a harmless strengthening by a fixed constant.  Therefore $R_s(a)>0$ uniformly in $a$.
\end{proof}

Suppose more explicitly that
\[
  H=p^{1/d+\delta},
  \qquad
  0<\delta<\frac1{d(\Delta_d-1)}.
\]
Then $E_d(p,H)\asymp p^{1/d}$, so \eqref{eq:additive-basis-condition} is satisfied for every fixed integer
\begin{equation}\label{eq:s-first-branch}
  s>\frac1\delta.
\end{equation}
By comparison, the same pointwise Fourier argument based on the best of
the classical and VMVT Weyl estimates requires
\begin{equation}\label{eq:s-standard-comparison}
  s>\frac{\Delta_d}{d\delta}.
\end{equation}
Thus, in the genuinely shortest part of the interval range, the new Weyl
estimate reduces by a factor $\Delta_d/d$ the number of variables supplied by
this deterministic additive-basis method; for $d\ge6$ the factor is $d-1$.
We do not claim that \eqref{eq:s-first-branch} is optimal among methods using
higher moments or additional algebraic information.

\section{Diophantine applications}\label{sec:diophantine-applications}

\subsection{Small fractional parts of general polynomials}

We use the standard large-multiple lemma from Baker's treatment of small fractional parts; see \cite[Theorem~2.2]{Baker1986} and the proof of \cite[Theorem~1]{Baker2016}.

\begin{lemma}[Baker]\label{lem:large-multiple}
Let $P(X)=\alpha_kX^k+\cdots+\alpha_1X$ and let $M\ge2$.  If
\[
  \|P(n)\|>M^{-1}
  \qquad(1\le n\le N),
\]
then
\begin{equation}\label{eq:large-multiple}
  \sum_{m=1}^{M}|g_k(m\boldsymbol\alpha;N)|\gg N.
\end{equation}
\end{lemma}

\begin{proof}[Proof of \cref{thm:small-fractional-parts-intro}]
Put $\Delta=\Delta_k$.  Suppose, to the contrary, that
\begin{equation}\label{eq:no-small-fractional-part}
  \|P(n)\|>N^{-1/\Delta+\varepsilon}
  \qquad(1\le n\le N).
\end{equation}
Let
\[
  M=\left\lfloor N^{1/\Delta-\varepsilon/2}\right\rfloor.
\]
By \cref{lem:large-multiple}, there is an integer $1\le m\le M$ such that
\begin{equation}\label{eq:large-multiple-one}
  |g_k(m\boldsymbol\alpha;N)|\gg N/M.
\end{equation}
After a harmless adjustment of the small parameters, this lies above the threshold in \cref{cor:combined-inverse}.  Hence there are integers $q,a_1,\ldots,a_k$ satisfying
\begin{equation}\label{eq:small-frac-q}
  q\ll M^kN^\eta,
  \qquad
  \|qm\alpha_j\|\ll M^kN^{-j+\eta}
  \quad(1\le j\le k)
\end{equation}
for arbitrarily small fixed $\eta>0$.

Set $n=qm$.  Since $\Delta\ge k+1$ for $k\ge3$,
\[
  n\ll M^{k+1}N^\eta\le N.
\]
Moreover,
\begin{align*}
  \|\alpha_jn^j\|
  &\le n^{j-1}\|qm\alpha_j\|\\
  &\ll M^{(k+1)(j-1)+k}N^{-j+j\eta}
  \ll N^{-1/\Delta-\varepsilon/4}
\end{align*}
when $\eta$ is sufficiently small.  Summing over $j$ contradicts \eqref{eq:no-small-fractional-part} for large $N$.
\end{proof}

For $k\ge6$, this gives the exponent $1/[k(k-1)]$ for arbitrary intermediate coefficients.  Baker obtained the same exponent for the binomial phase $\alpha_kn^k+\alpha_1n$, whereas his 2016 theorem for a general polynomial used $1/[2k(k-1)]$ for $k\ge8$ \cite[Theorems~1 and~2(a)]{Baker2016}.

\subsection{Polynomial fractional parts at prime arguments}

We now explain why Baker's Harman-sieve argument \cite{BakerPrimes2017} accepts the improved inverse threshold without further analytic changes.

\begin{proof}[Proof of \cref{thm:prime-fractional-parts-intro}]
Baker associates to a degree-$k$ polynomial a denominator parameter $J(f)$.  For a general polynomial his choice is
\[
  J(f)=2^{k-1}\quad(k\le7),
  \qquad
  J(f)=2k(k-1)\quad(k\ge8),
\]
while for the binomial $\alpha X^k+\beta$ he uses $J(f)=k(k-1)$ from degree six onward.  His theorem states that, for $k\ge4$, every
\[
  \nu<\frac{0.4079}{J(f)}
\]
is admissible.

The parameter $J(f)$ first enters through \cite[Lemma~5]{BakerPrimes2017}.  In the nonclassical case that lemma is applied with one selected large Weyl sum ($M=1$), and its simultaneous approximation conclusion is precisely the conclusion of Baker's 2016 inverse theorem.  Translating the interval and multiplying the phase by the selected integer do not affect the uniformity of \cref{cor:combined-inverse}.  Thus \cite[Lemma~5]{BakerPrimes2017} remains valid for a general polynomial with
\begin{equation}\label{eq:new-J-primes}
  J=\Delta_k.
\end{equation}

For completeness, we record the parameter check in the part of Baker's proof where $J$ is subsequently used.  Put
\[
  \rho=\frac{0.4079}{J}.
\]
The Type~I and Type~II ranges involve the product $J\rho=0.4079$, which is unchanged by \eqref{eq:new-J-primes}, together with $J\ge k+1$.  In the new cases $k\ge6$, one has $J=k(k-1)$, and the remaining numerical requirements used in the Type~I and Type~II estimates include
\[
  \rho\left(\frac52-\frac{3}{2k}\right)<\frac1{2k},
  \qquad
  \rho<\frac{3k}{20k+5};
\]
both are immediate for $\rho=0.4079/[k(k-1)]$.  The complete-sum estimates, the Type~I/II decompositions, and the final Harman-sieve calculation are otherwise unchanged.  Baker's proof therefore yields \eqref{eq:prime-fractional-parts} for every $\nu$ in \eqref{eq:prime-nu-range}.
\end{proof}

The numerical improvement begins in degree six:
\[
\begin{array}{c|ccc|c}
  k&6&7&8&k\ge8\\ \hline
  \text{Baker's general denominator}&32&64&112&2k(k-1)\\
  \text{new denominator}&30&42&56&k(k-1).
\end{array}
\]
Thus the prime-argument exponent is asymptotically doubled for general polynomials.  This conclusion concerns arbitrary intermediate coefficients; specialized monomial estimates may be stronger in other regimes.

\section{Further consequences and remarks}\label{sec:further-consequences}

\subsection{Maximal operators and denominators}

Baker--Chen--Shparlinski \cite{BakerChenShparlinski2024} use their large-value inverse theorem as the entrance point to a refined prime-power factorisation of the common denominator and then to estimates for maximal Weyl operators.  In every part of their argument whose lower amplitude range is limited only by \cite[Lemma~2.6]{BakerChenShparlinski2024}, one may replace
\[
  D_k=\min\{2^{k-1},2k(k-1)\}
  \quad\text{by}\quad
  \Delta_k=\min\{2^{k-1},k(k-1)\}.
\]
In particular, the factorisation and large-value counting statements in their Lemmas~2.7 and~2.9 extend to the enlarged range
$A>N^{1-1/\Delta_k+\varepsilon}$.  We do not reproduce the resulting maximal-operator formulas here, since their proofs and conclusions are unchanged apart from this threshold substitution.

\subsection{Exceptional sets and local mean values}

The same replacement propagates to two related developments.  The covering argument in Baker--Chen--Shparlinski's work on large Weyl sums and Hausdorff dimension \cite{BakerChenShparlinskiHausdorff2022} uses the large-value structure only above its inverse-theorem threshold; hence its corresponding upper bounds extend from the range governed by $D_k$ to that governed by $\Delta_k$.  Likewise, the local mean value arguments of Brandes--Chen--Shparlinski \cite{BrandesChenShparlinski2024} invoke the refined denominator description before performing their box counting.  Replacing that input by the enlarged version above improves the associated admissible moment ranges.  These are direct transfers rather than new arguments, and we leave their numerous parameterized formulations in the notation of the cited papers.

\subsection{The critical threshold and a possible next step}

The exponent $k(k-1)$ is the critical moment exponent for the degree-$(k-1)$ Vinogradov system.  The all-cuts argument supplies $\asymp N$ sampling points, while the lower-coefficient frequency rectangle has volume $N^{k(k-1)/2}$.  At the critical moment, the sampling inequality and the fixed-leading-coefficient maximal estimate give
\[
  NA^{k(k-1)}\ll N^{k(k-1)+\varepsilon},
\]
which forces a collision at exactly the threshold
$A>N^{1-1/[k(k-1)]+\varepsilon}$.

The orbit centres are, however, not arbitrary points of $\mathbb T^{k-1}$: they lie on the polynomial translation orbit
$m\mapsto\pi(\mathcal T_m\boldsymbol\alpha)$.  Improving the threshold further would require exploiting this structure before the first collision.  Expanding an orbit average and stratifying by the first nonvanishing power-sum difference produces a lower-degree Weyl phase in the orbit parameter.  This suggests an orbit-restricted broad--narrow or decoupling refinement, with the narrow case feeding back additional rational structure.  We do not pursue that iteration here.

\subsection{Uniformity in the degree}

All estimates proved in this paper, unless explicitly stated otherwise, have the following order of quantifiers: the degree and the displayed small parameter are fixed, and then the estimates hold uniformly as the lengths, coefficients, moduli, and intervals vary.  No uniformity is asserted for a degree tending to infinity with the length or modulus.

\section*{Statement on the use of AI}

ChatGPT 5.6 Pro was used during exploratory work and in preparing an initial draft, including algebraic calculations, literature searches, exposition, and \LaTeX{} preparation.  The author is responsible for checking every argument and for the mathematical content of any submitted version.

\end{document}